\documentclass[reqno]{amsart}
\usepackage{amssymb,epsfig}

\newtheorem{theorem}{Theorem}
\newtheorem{lemma}[theorem]{Lemma}

\theoremstyle{definition}

\newtheorem{example}[theorem]{Example}

\newtheorem{notation}{Notation}

\theoremstyle{remark}
\newtheorem{remark}[theorem]{Remark}

\begin{document}

\title{On algorithms for testing positivity of symmetric polynomial functions}

\author{Vlad Timofte}
\address{Institute of Mathematics ``Simion Stoilow'' of the Romanian Academy, P.O. Box 1-764, RO-014700 Bucharest, Romania}
\email{vlad.timofte@imar.ro}
\author{Aida Timofte}
\address{Institute of Mathematics ``Simion Stoilow'' of the Romanian Academy, P.O. Box 1-764, RO-014700 Bucharest, Romania}
\email{aida.timofte@imar.ro}

\date{November 9, 2020}

\begin{abstract}
We show that positivity on $\mathbb{R}_+^n$ and on $\mathbb{R}^n$ of real symmetric polynomials of degree at most $p$ in $n\ge2$ variables is solvable by algorithms running in $\mathrm{poly}(n)$ time. For real symmetric quartics, we find explicit discriminants and related Maple algorithms running in $\mathrm{lin}(n)$ time.
\end{abstract}

\maketitle

\section{Introduction}\label{s.introduction}

In this paper we deal with special cases of the quantifier elimination problems
\[\begin{array}{ll} f(x)\ge0&\forall x\in\mathbb{R}_+^n,\\ f(x)\ge0&\forall x\in\mathbb{R}^n, \end{array}\]
which will be referred to as $\mathrm{QE}_+(f)$ and $\mathrm{QE}(f)$ respectively, for real polynomials $f\in\mathbb{R}[X_1,\dots,X_n]$.

Let us denote\footnote{We follow the notations from~\cite{timofte1,timofte2,timofte3}.} by $\Sigma_p^{[n]}$ the vector space of all real symmetric polynomials of degree at most $p\in\mathbb{N}$ in $n\ge2$ indeterminates, and by ${\mathcal H}_p^{[n]}$ its subspace consisting of $p$-homogeneous polynomials (real symmetric $n$-ary $p$-forms). For symmetric cubics $f\in{\mathcal H}_3^{[n]}$, it is known from~\cite[Th.\,3.7]{c.l.r} that $\mathrm{QE}_+(f)$ holds, if and only if
\begin{equation}\label{e.k}f(1_k,0_{n-k})\ge0\quad\mbox{for every\, }k\in\{1,2,\dots,n\},\end{equation}
where $1_k:=(1,\dots,1)\in\mathbb{R}^k$ and $0_k:=(0,\dots,0)\in\mathbb{R}^k$ for every $k$. This equivalence is no longer true for higher degree, but still holds for $f\in{\mathcal H}_4^{[n]}$ (\emph{symmetric quartic}) under the additional condition (see~\cite[Th.\,19(2)]{timofte2})
\[f(1,-1,0_{n-2})\le0.\]
For $\mathrm{QE}_+(f)$ and $\mathrm{QE}(f)$, the mere existence of an equivalent boolean combination of polynomial inequalities in the coefficients of $f$ follows by the Tarski-Seidenberg principle (see~\cite{b.c.r,jacobson} for details). For every fixed degree upper bound $p\in\mathbb{N}^*$, the problems $\mathrm{QE}_+(f)$ and $\mathrm{QE}(f)$ for $\deg(f)\le p$ are well-known to be unsolvable in $\mathrm{poly}(n)$ time (see for instance~\cite{c.c.s}). Existing algorithms cannot be efficient for the problems $\mathrm{QE}_+(f)$ and $\mathrm{QE}(f)$ restricted to symmetric polynomials, since they perform the same operations as in the general case. Nonetheless, in the symmetric case specific algorithms running in $\mathrm{poly}(n)$ time can be designed.

For symbolic quartic $f\in{\mathcal H}_4^{[n]}$, we compute explicit systems of discriminants for both $\mathrm{QE}_+(f)$ and $\mathrm{QE}(f)$. By combining and strengthening some results from~\cite{timofte1,timofte2}, both problems are reduced to equivalent finite systems of univariate polynomial inequalities of degree at most $4$, for which we find explicit discriminants. The resulting algorithms $\mathrm{QE}4_+$ and $\mathrm{QE}4$ from Sections~\ref{ss.algorithm1} and~\ref{ss.algorithm2} run in $\mathrm{lin}(n)$ time. Related numerical and theoretical examples are discussed in Sections~\ref{ss.algorithm1} and~\ref{ss.algorithm2}.

\section{Existence of efficient algorithms in the symmetric case}\label{s.efficient}

\begin{theorem}[efficient algorithms]\label{t.efficient}
There exists an algorithm solving for every fixed degree upper bound $p\in\mathbb{N}^*$ the problem $\mathrm{QE}_+(f)/\mathrm{QE}(f)$ for arbitrary $f\in\Sigma_p^{[n]}$ in $\mathrm{poly}(n)$ time. The statement also holds for the corresponding problems defined with strict inequalities.
\end{theorem}

\begin{proof}
For all $p,q\in\mathbb{N}^*$, set ${\mathcal P}_p^{[q]}:=\{g\in\mathbb{R}[X_1,\dots,X_q]\,|\,\deg(g)\le p\}$ and
\[R_{p,q}:=\bigg\{r=(r_1,\dots,r_p)\in(\mathbb{N}^*)^p\,\bigg|\,\sum_{i=1}^pr_i=q\bigg\},\quad R_{p,q}^+:=\bigcup_{k=1}^qR_{p,k}.\]
\emph{The case of} $\mathrm{QE}_+(f)$. There is an algorithm\footnote{Several such algorithms are known, but their structure is not relevant here.} ${\mathcal A}_+$ which solves the problem $\mathrm{QE}_+(g)$ for arbitrary $g\in{\mathcal P}_p^{[q]}$ by performing $N_+(p,q)$ operations. Let us fix $p\in\mathbb{N}^*$ and set $\bar p:=\max\{\lfloor\frac p2\rfloor,1\}$. For all $f\in\Sigma_p^{[n]}$ and $r\in R_{\bar p,n}^+$, let us define $f_r\in{\mathcal P}_p^{[\bar p]}$ by
\[f_r(u_1,\dots,u_{\bar p}):=f(u_1\!\cdot\!1_{r_1},\,\dots\,,u_{\bar p}\!\cdot\!1_{r_{\bar p}},0_{r'}),\]
where $r'=n-\sum_{i=1}^pr_i$. According to~\cite[Cor.\,2.1(1)]{timofte1}, for every $f\in\Sigma_p^{[n]}$ we have
\begin{equation}\label{e.QE+}\mathrm{QE}_+(f)\mbox{ holds}\iff\mathrm{QE}_+(f_r)\mbox{ holds for every }r\in R_{\bar p,n}^+.\end{equation}
Let us consider the algorithm $\tilde{\mathcal A}_+$ which solves the problem $\mathrm{QE}_+(f)$ by running ${\mathcal A}_+$ for all $\mathrm{QE}_+(f_r)$ ($r\in R_{\bar p,n}^+$). The number of operations performed by $\tilde{\mathcal A}_+$ on $f$ is
\[N(\tilde{\mathcal A}_+,f)=N_+(p,\bar p)\cdot\mathrm{card}(R_{\bar p,n}^+)=N_+(p,\bar p)\binom{n}{\bar p}.\]
Hence for fixed $p$, the problem $\mathrm{QE}_+(f)$ for $f\in\Sigma_p^{[n]}$ is solvable in $O(n^{\bar p})$ time.\\
\emph{The case of} $\mathrm{QE}(f)$. Let us note that $\mathrm{QE}(f)$ cannot hold if $\deg(f)$ is odd. There is an algorithm ${\mathcal A}$ which solves the problem $\mathrm{QE}(g)$ for arbitrary $g\in{\mathcal P}_p^{[q]}$ by performing $N(p,q)$ operations. Let us fix an even integer $p\ge2$ and set $\bar p:=\max\{\frac p2,2\}$. For all $f\in\Sigma_p^{[n]}$ and $r\in R_{\bar p,n}$, let us define $f_r\in{\mathcal P}_p^{[\bar p]}$ by
\[f_r(u_1,\dots,u_{\bar p}):=f(u_1\!\cdot\!1_{r_1},\,\dots\,,u_{\bar p}\!\cdot\!1_{r_{\bar p}}).\]
According to~\cite[Cor.\,2.1(2)]{timofte1}, for every $f\in\Sigma_p^{[n]}$ we have the equivalence
\begin{equation}\label{e.QE}\mathrm{QE}(f)\mbox{ holds}\iff\mathrm{QE}(f_s)\mbox{ holds for every }r\in R_{\bar p,n}.\end{equation}
Let us consider the algorithm $\tilde{\mathcal A}$ which solves the problem $\mathrm{QE}(f)$ by running $\mathcal A$ for all $\mathrm{QE}(f_r)$ ($r\in R_{\bar p,n}$). The number of operations performed by $\tilde{\mathcal A}$ on $f$ is
\[N(\tilde{\mathcal A},f)=N(p,\bar p)\cdot\mathrm{card}(R_{\bar p,n})=N(p,\bar p)\binom{n-1}{\bar p-1}.\]
Hence for fixed $p$, the problem $\mathrm{QE}(f)$ for $f\in\Sigma_p^{[n]}$ is solvable in $O(n^{\bar p-1})$ time.\\
\emph{The case of strict inequalities}. For both problems the proof using Corollary~2.1 from~\cite{timofte1} is similar to the above.
\end{proof}

\section{The problem $\mathrm{QE}_+(f)$ in ${\mathcal H}_4^{[n]}$}\label{s.QE1}

\subsection{Finite test-sets for $\mathrm{QE}_+(f)$}\label{ss.test}

According to (\ref{e.QE+}), for every symmetric quartic $f\in{\mathcal H}_4^{[n]}$ the problem $\mathrm{QE}_+(f)$ reduces to the quantifier elimination problems
\[f(u\!\cdot\!1_r,v\!\cdot\!1_s,0_{n-r-s})\ge0\quad\forall u,v\ge0,\]
considered for all $r,s\in\mathbb{N}^*$, with $r+s\le n$. Since $p=4$ and $\bar p=2$, the algorithm $\tilde{\mathcal A}_+$ described in the proof of Theorem~\ref{t.efficient} solves $\mathrm{QE}_+(f)$ in $O(n^2)$ time. Theorem~\ref{t.quartics} below will lead to the algorithm $\mathrm{QE}4_+$, which solves $\mathrm{QE}_+(f)$ in $\mathrm{lin}(n)$ time.

For every $f\in{\mathcal H}_4^{[n]}$ we assume the representation
\begin{equation}\label{e.rep}f=aP_4+bP_3P_1+cP_2^2+dP_2P_1^2+eP_1^4\quad(a,b,c,d,e\in\mathbb{R}),\end{equation}
where $P_k$ denotes the $k$th symmetric power sum $P_k(x_1,\dots,x_n)=\sum_{j=1}^nx_j^k$. For all $f\in{\mathcal H}_4^{[n]}$ and $(r,s)\in\mathbb{N}^*\times\mathbb{N}^*$ with $r+s\le n$, let us define
\begin{eqnarray*}
&&f_{r,s}:\mathbb{R}\rightarrow\mathbb{R},\quad f_{r,s}(t)=f(t\!\cdot\!1_r,1_s,0_{n-r-s}),\\
&&f_{r,s}^\partial:\mathbb{R}\setminus\{1\}\rightarrow\mathbb{R},\quad f_{r,s}^\partial(t)=\frac{\left(\frac{\partial f}{\partial x_1}-\frac{\partial f}{\partial x_{r+1}}\right)(t\!\cdot\!1_r,1_s,0_{n-r-s})}{t-1}.
\end{eqnarray*}
An easy computation shows that
\[f_{r,s}^\partial(t)=4a(t^2+t+1)+3b(t+1)(rt+s)+4c(rt^2+s)+2d(rt+s)^2.\]
Hence $f_{r,s}^\partial$ is the restriction of a polynomial function (for which we use the same notation) with\ $\deg(f_{r,s}^\partial)\le2$. We have the obvious identities
\begin{equation}\label{e.dual}f_{r,s}(t)=t^4f_{s,r}\mbox{$\left(\frac{1}{t}\right)$}\mbox{ \ and \ } f_{r,s}^\partial(t)=t^2f_{s,r}^\partial\mbox{$\left(\frac{1}{t}\right)$, \ for every }t\in\mathbb{R}^*.\end{equation}

Let $Z^f\subset\mathbb{N}^*\times\mathbb{N}^*$ denote the finite set consisting of all points of the form
\[\left\{\begin{array}{ll}(k,1),(1,k),(k,n-k)&\mbox{ for }k\in\{1,2,\dots,n-1\},\mbox{ if }a>0>b,\\
(k,1),(1,k)&\mbox{ for }k\in\{1,2,\dots,n-1\},\mbox{ else}.\end{array}\right.\]
Let us consider the finite sets
\begin{eqnarray*}
M^f\!\!\!&:=&\!\!\!\{(r,s)\in Z^f\,|\, f_{r,s}^\partial\mbox{ is non-constant}\},\\
T_{r,s}^f\!\!\!&:=&\!\!\!\{(t\!\cdot\!1_r,1_s,0_{n-r-s})\in\mathbb{R}^n\,|\,t>1,\,f_{r,s}^\partial(t)=0\}\mbox{ for every }(r,s)\in M^f,\\
T^f\!\!\!&:=&\!\!\!\{(1_k,0_{n-k})\,|\,k\in\{1,2,\dots,n\}\}\cup\bigcup_{(r,s)\in M^f}T_{r,s}^f.
\end{eqnarray*}
The following theorem is a combination of three known results from~\cite[Ths.\,13,14]{timofte2} for the equivalences (\ref{e.equivalence1}) and (\ref{e.equivalence2}), and from~\cite[Cor.\,5.6]{timofte1} for (\ref{e.equivalence3}).

\begin{theorem}[finite test sets]\label{t.quartics}
For every $f\in{\mathcal H}_4^{[n]}$, we have the equivalences:
\begin{eqnarray}
\label{e.equivalence1}\mathrm{QE}_+(f)\mbox{ holds}\!\!\!&\iff&\!\!\!\mathrm{QE}_+(f_{r,s})\mbox{ holds for every }(r,s)\in Z^f\\
\label{e.equivalence2}&\iff&\!\!\!f(\xi)\ge0\mbox{\, for every\, }\xi\in T^f.
\end{eqnarray}
\end{theorem}

\begin{proof}
For every $(r,s)\in\mathbb{N}^*\times\mathbb{N}^*$ with $r+s\le n$, by (\ref{e.dual}) we deduce that
\[\mathrm{QE}_+(f_{r,s})\iff\mathrm{QE}_+(f_{s,r}).\]
Thus by the definition of $Z^f$ we see that (\ref{e.equivalence1}) is a restatement of~\cite[Th.\,14]{timofte2}.\\
For (\ref{e.equivalence2}), we only need to prove the implication ``$\Leftarrow$''. Fix $(r,s)\in Z^f$. We claim that $f_{r,s}\ge0$ on $[1,\infty[$. As $f$ is homogeneous, this holds if and only if $g\ge0$, where
\[g:[1,\infty[\,\rightarrow\mathbb{R},\quad g(t)=f\left(\frac{t}{rt+s}\!\cdot\!1_r,\,\frac{1}{rt+s}\!\cdot\!1_s,\,0_{n-r-s}\right).\]
Assume that $g\ge0$ does not hold. Some trivial computations show that
\begin{eqnarray}\label{e.lim}&&g(1)=\frac{f(1_{r+s},0_{n-r-s})}{(r+s)^4}\ge0,\quad\lim_{t\to\infty}g(t)=\frac{f(1_r,1_s,0_{n-r-s})}{r^4}\ge0,\\
\label{e.g'}&&g'(t)=\frac{rs(t-1)}{(rt+s)^5}f_{r,s}^\partial(t) \ \mbox{ for every }t\ge1.\end{eqnarray}
By our assumption on $g$ and (\ref{e.lim}), we see that $g$ has a global minimum at some $\theta>1$ and $g(\theta)<0$. By (\ref{e.g'}) it follows that $f_{r,s}^\partial(\theta)=0$, hence that $(r,s)\in M^f$ (otherwise, $g$ is constant), and finally that $\xi:=(\theta\cdot1_r,1_s,0_{n-r-s})\in T_{r,s}^f$. We thus get $(r\theta+s)^4g(\theta)=f_{r,s}(\theta)=f(\xi)\ge0$, a contradiction. Our claim is proved. We conclude that $f_{r,s}\ge0$ on $[1,\infty[$ for every $(r,s)\in Z^f$. Now using (\ref{e.dual}) and the symmetry of $Z^f$ shows that $f_{r,s}\ge0$ on $[0,1]$ for every $(r,s)\in Z^f$.
\end{proof}


For a weaker version of (\ref{e.equivalence1}) and for other results on $\mathrm{QE}_+(f)$ in the setting of ternary ($n=3$) even symmetric octics, we refer the reader to~\cite{harris}.

\subsection{Discriminants for $\mathrm{QE}_+(f)$}\label{ss.discriminants1}

The equivalence (\ref{e.equivalence1}) reduces $\mathrm{QE}_+(f)$ to finitely many univariate problems $\mathrm{QE}_+(f_{r,s})$. For quartics $f$ with numerical coefficients, one can apply the classical Sturm theorem to each $\mathrm{QE}_+(f_{r,s})$. Nonetheless, in the case of symbolic/literal coefficients this approach is inconvenient, cannot deal with theoretical problems (such as those from Examples~\ref{ex.3} and \ref{ex.4}), and cannot lead to a system of discriminants for $\mathrm{QE}_+(f)$. Since our approach to $\mathrm{QE}_+(f)$ is based on the equivalence (\ref{e.equivalence2}), we next characterize the inequalities $f|_{T_{r,s}^f}\ge0$ for $(r,s)\in M^f$.

Some easy computations show that $f(1,-1,0_{n-2})=2(a+2c)$ and
\[\left\{\begin{array}{l}f_{r,s}(t)=A_{r,s}t^4+B_{r,s}t^3+C_{r,s}t^2+D_{r,s}t+E_{r,s},\\
f_{r,s}^\partial(t)=\alpha_{r,s}t^2+\beta_{r,s}t+\gamma_{r,s},\end{array}\right.\]
where
\[\left\{\begin{array}{l}A_{r,s}:=r[a+(b+c)r+dr^2+er^3],\\ B_{r,s}:=rs(b+2dr+4er^2),\\ C_{r,s}:=rs[2c+d(r+s)+6ers],\\ D_{r,s}:=rs(b+2ds+4es^2),\\ E_{r,s}:=s[a+(b+c)s+ds^2+es^3],\end{array}\right.\quad \left\{\begin{array}{l}\alpha_{r,s}:=4a+(3b+4c)r+2dr^2,\\ \beta_{r,s}:=4a+3b(r+s)+4drs,\\ \gamma_{r,s}:=4a+(3b+4c)s+2ds^2.\end{array}\right.\]
Therefore, for arbitrarily given real symbolic polynomials
\[\left\{\begin{array}{l}F(t)=At^4+Bt^3+Ct^2+Dt+E,\\ g(t)=\alpha t^2+\beta t+\gamma, \ \ g\mbox{
is non-constant},\end{array}\right.\]
we need to compute explicit discriminants for the problem
\begin{equation}\label{e.statement1}F(\theta)\ge0\,\mbox{ for every root }\,\theta>0\mbox{ of }g.\end{equation}
As $\deg(g)\in\{1,2\}$, analyzing the equation $g(t)=0$ in both possible cases ($\alpha\ne0$ or $\alpha=0\ne\beta$) shows the equivalence
\[(\ref{e.statement1})\iff(\ref{e.deg1})\mbox{\ \bf or\ }(\ref{e.complex})\mbox{\ \bf or\ }(\ref{e.deg2}),\]
where
\begin{eqnarray}\label{e.deg1}&&\left\{\begin{array}{l}\alpha=0,\,\beta\ne0,\\ \beta\gamma\ge0\mbox{\ \bf or\ }F\left(-\frac{\gamma}{\beta}\right)\ge0,\end{array}\right.\\[2mm]
\label{e.complex}&&\Delta:=\beta^2-4\alpha\gamma<0,\\[2mm]
\label{e.deg2}&&\left\{\begin{array}{l}\alpha\ne0,\,\Delta\ge0,\\ \alpha(\beta-\sqrt{\Delta})\ge0\mbox{\ \bf or\ }F\left(\frac{-\beta+\sqrt{\Delta}}{2\alpha}\right)\ge0,\\ \alpha(\beta+\sqrt{\Delta})\ge0\mbox{\ \bf or\ }F\left(\frac{-\beta-\sqrt{\Delta}}{2\alpha}\right)\ge0.\end{array}\right.\end{eqnarray}

For the elimination of $\sqrt{\Delta}$ from (\ref{e.deg2}) we need the following lemma.

\begin{lemma}\label{l.computer1}
Let $u,v,\delta\in\mathbb{R}$, with $\delta\ge0$. Then
\begin{eqnarray}
\label{e.equiv1}u+v\sqrt{\delta}\ge0\!\!\!&\iff&\!\!\!u,v\ge0\mbox{ \rm\bf or }\rho,u\ge0\mbox{ \rm\bf or }\rho\le0\le v\\
\label{e.equiv2}&\iff&\!\!\!u,v\ge0\mbox{ \rm\bf or }\rho,u\ge0\ge v\mbox{ \rm\bf or }\rho,u\le0\le v,
\end{eqnarray}
where $\rho:=u^2-v^2\delta$.
\end{lemma}

\begin{proof}
The proof is routine.
\end{proof}

Now assume that $\alpha\ne0$ and $\Delta\ge0$ as in (\ref{e.deg2}). Hence $F\left(\frac{-\beta\pm\sqrt{\Delta}}{2\alpha}\right)=\frac{P\pm Q\sqrt{\Delta}}{2\alpha^4}$ for some polynomial expressions $P,Q$, in the coefficients of $F$ and $g$. We next need to compute the expressions $P,Q$, and $P^2-Q^2\Delta$, since these will characterize the inequalities $F\left(\frac{-\beta\pm\sqrt{\Delta}}{2\alpha}\right)\ge0$ from (\ref{e.deg2}) via the equivalence (\ref{e.equiv1}). By Taylor's formula, we see that\ $\frac{P+Q\sqrt{\Delta}}{2\alpha^4}=F\left(\frac{-\beta+\sqrt{\Delta}}{2\alpha}\right)=\sum_{k=0}^4F^{(k)}\!\left(\frac{-\beta}{2\alpha}\right)\frac{(\sqrt{\Delta})^k}{k!(2\alpha)^k}$.\ As an easy computation shows, we may choose\footnote{$Q$ is not unique if $\Delta=0$, but we still can choose it as in (\ref{e.Q}).}
\begin{eqnarray}\label{e.P}
&&\!\begin{array}{ll}\!\!P=&\!\!\!A(\beta^4-4\beta^2\alpha\gamma+2\alpha^2\gamma^2)-B\beta\alpha(\beta^2-3\alpha\gamma)+C\alpha^2(\beta^2-2\alpha\gamma)\\ &-D\beta\alpha^3+2E\alpha^4,\end{array}\\
\label{e.Q}&&Q=-A\beta(\beta^2-2\alpha\gamma)+B\alpha(\beta^2-\alpha\gamma)-C\beta\alpha^2+D\alpha^3.
\end{eqnarray}
Let us observe that $R:=\frac{P^2-Q^2\Delta}{4\alpha^4}=\alpha^4F\left(\frac{-\beta+\sqrt{\Delta}}{2\alpha}\right)F\left(\frac{-\beta-\sqrt{\Delta}}{2\alpha}\right)$ is the resultant of the polynomials $g$ and $F$, and so
\begin{equation}\label{e.resultant1}R=\left|\begin{array}{cccccc}A&B&C&D&E&0\\ 0&A&B&C&D&E\\ \alpha&\beta&\gamma&0&0&0\\ 0&\alpha&\beta&\gamma&0&0\\ 0&0&\alpha&\beta&\gamma&0\\ 0&0&0&\alpha&\beta&\gamma\end{array}\right|.\end{equation}
By Lemma~\ref{l.computer1} we get the equivalences
\begin{eqnarray}
\label{e.expl1}P\pm Q\sqrt{\Delta}\ge0\!\!\!&\iff&\!\!\!P,\pm Q\ge0\mbox{ \bf or }R,P\ge0\mbox{ \bf or }R\le0\le\pm Q,\\
\label{e.expl2}\alpha(\beta\pm\sqrt{\Delta})\ge0\!\!\!&\iff&\!\!\!\pm\alpha,\pm\beta\ge0\mbox{ \bf or }\pm\beta,\pm\gamma\le0
\end{eqnarray}
(their proof uses (\ref{e.equiv1}) for (\ref{e.expl1}) and (\ref{e.equiv2}) for (\ref{e.expl2})). Hence (\ref{e.deg2}) is equivalent to
\begin{equation}\label{e.discriminants1}\left\{\begin{array}{l}\alpha,\beta\le0\mbox{ \bf or }\beta,\gamma\ge0\mbox{ \bf or }P,Q\ge0\mbox{ \bf or }R,P\ge0\mbox{ \bf or }R\le0\le Q,\\ \alpha,\beta\ge0\mbox{ \bf or }\beta,\gamma\le0\mbox{ \bf or }P\ge0\ge Q\mbox{ \bf or }R,P\ge0\mbox{ \bf or }R,Q\le0,\end{array}\right.\end{equation}
provided that $\alpha\ne0\le\Delta$.

\begin{notation}\label{n.QE1}
For $(F,g)=(f_{r,s},f_{r,s}^\partial)$, we write the conditions $(\ref{e.statement1})$--$(\ref{e.deg2}),(\ref{e.discriminants1})$ and the polynomials $P,Q,R$ from $(\ref{e.P})$-$(\ref{e.resultant1})$, as $(\ref{e.statement1})_{r,s}$-$(\ref{e.deg2})_{r,s},(\ref{e.discriminants1})_{r,s}$, and $P_{r,s},Q_{r,s},R_{r,s}$. Set $\Delta_{r,s}:=\beta_{r,s}^2-4\alpha_{r,s}\gamma_{r,s}$.
\end{notation}

As the expressions of $P_{r,s},Q_{r,s},R_{r,s}$ are exceedingly long (filling several pages), we will not reproduce them here. Let us mention that a Maple\,15 computation has shown that $a(r^2-rs+s^2)+crs(r+s)$ is a factor of $R_{r,s}$.

We are now able to state and prove our first main result.

\begin{theorem}[$\mathrm{QE}4_+$ algorithm]\label{t.discriminants1}
Let $f\in{\mathcal H}_4^{[n]}$. Then $\mathrm{QE}_+(f)$ is equivalent to
\[\left\{\begin{array}{l}(\ref{e.k})\mbox{ holds for }f,\\ (\ref{e.discriminants1})_{r,s}\mbox{ holds for every $(r,s)\in Z^f$, such that }\Delta_{r,s}\ge0\ne\alpha_{r,s}.\end{array}\right.\]
\end{theorem}

\begin{proof}
``$\Rightarrow$''. Obviously, $\mathrm{QE}_+(f)$ yields (\ref{e.k}). Let us fix $(r,s)\in Z^f$, such that $\Delta_{r,s}\ge0\ne\alpha_{r,s}$. By $\mathrm{QE}_+(f)$, we deduce that $(\ref{e.statement1})_{r,s}$ holds. Since both $(\ref{e.deg1})_{r,s}$ and $(\ref{e.complex})_{r,s}$ are not fulfilled, it follows that $f$ satisfies $(\ref{e.deg2})_{r,s}$, that is, $(\ref{e.discriminants1})_{r,s}$ holds.\\
``$\Leftarrow$''. Under the hypothesis of the implication, suppose that $f(\xi)<0$ for some $\xi\in T^f$. As (\ref{e.k}) holds, we have $\xi=(\theta\cdot\!1_r,1_s,0_{n-r-s})\in T_{r,s}^f$ for some $(r,s)\in M^f$ and $\theta>1$, with $f_{r,s}^\partial(\theta)=0$. It is easily seen that $(s,r)\in M^f$ and
\begin{equation}\label{e.switch1}\left\{\begin{array}{l}(\alpha_{s,r},\beta_{s,r},\gamma_{s,r})=(\gamma_{r,s},\beta_{r,s},\alpha_{r,s}),\quad\Delta_{s,r}=\Delta_{r,s}\ge0,\\
f_{s,r}^\partial\left(\frac{1}{\theta}\right)=\frac{f_{r,s}^\partial(\theta)}{\theta^2}=0,\quad f_{s,r}\left(\frac{1}{\theta}\right)=\frac{f_{r,s}(\theta)}{\theta^4}=\frac{f(\xi)}{\theta^4}<0.\end{array}\right.\end{equation}
Hence both $(\ref{e.deg2})_{r,s}$ and $(\ref{e.deg2})_{s,r}$ are false, since so are $(\ref{e.statement1})_{r,s}$ and $(\ref{e.statement1})_{s,r}$, by (\ref{e.switch1}). We must have $\alpha_{r,s}=0$, since otherwise $(\ref{e.discriminants1})_{r,s}$ holds and is equivalent to $(\ref{e.deg2})_{r,s}$. Similarly, using again (\ref{e.switch1}) yields $\gamma_{r,s}=0$. Thus $\beta_{r,s}\theta=f_{r,s}^\partial(\theta)=0$ forces $\beta_{r,s}=0$, and hence $f_{r,s}^\partial\equiv0$, a contradiction with $(r,s)\in M^f$. We conclude that $f(\xi)\ge0$ for every $\xi\in T^f$, which yields $\mathrm{QE}_+(f)$, by Theorem~\ref{t.quartics}.
\end{proof}

\begin{remark}\label{r.useful}
In Theorem~\ref{t.discriminants1}, for $r+s\le n$ the condition $(\ref{e.discriminants1})_{r,s}$ either holds or needs not be satisfied whenever \ $\alpha_{r,s}=0$ {\rm or} $\beta_{r,s}=0$ {\rm or} $\Delta_{r,s}<0$ {\rm or} $\alpha_{r,s},\beta_{r,s},\gamma_{r,s}\ge0$ {\rm or} $\alpha_{r,s},\beta_{r,s},\gamma_{r,s}\le0$ {\rm or} $P_{r,s},R_{r,s}\ge0$.
\end{remark}

\subsection{The algorithm $\mathrm{QE}4_+$ and examples}\label{ss.algorithm1}

Theorem~\ref{t.discriminants1} leads to the algorithm $\mathrm{QE}4_+$ for solving $\mathrm{QE}_+(f)$ (see Figure~\ref{f.qe4+}). For $n\ge3$, the algorithm performs $n$ tests for the condition (\ref{e.k}) and at most $3n-6\ge\mathrm{card}(Z^f)$ sets of computations/tests (of the \emph{same complexity} for all $(r,s)\in Z^f$) for the conditions $(\ref{e.discriminants1})_{r,s}$. Hence $\mathrm{QE}4_+$ solves $\mathrm{QE}_+(f)$ in $\mathrm{lin}(n)$ time.

\begin{figure}
\begin{center} \includegraphics[width=12cm]{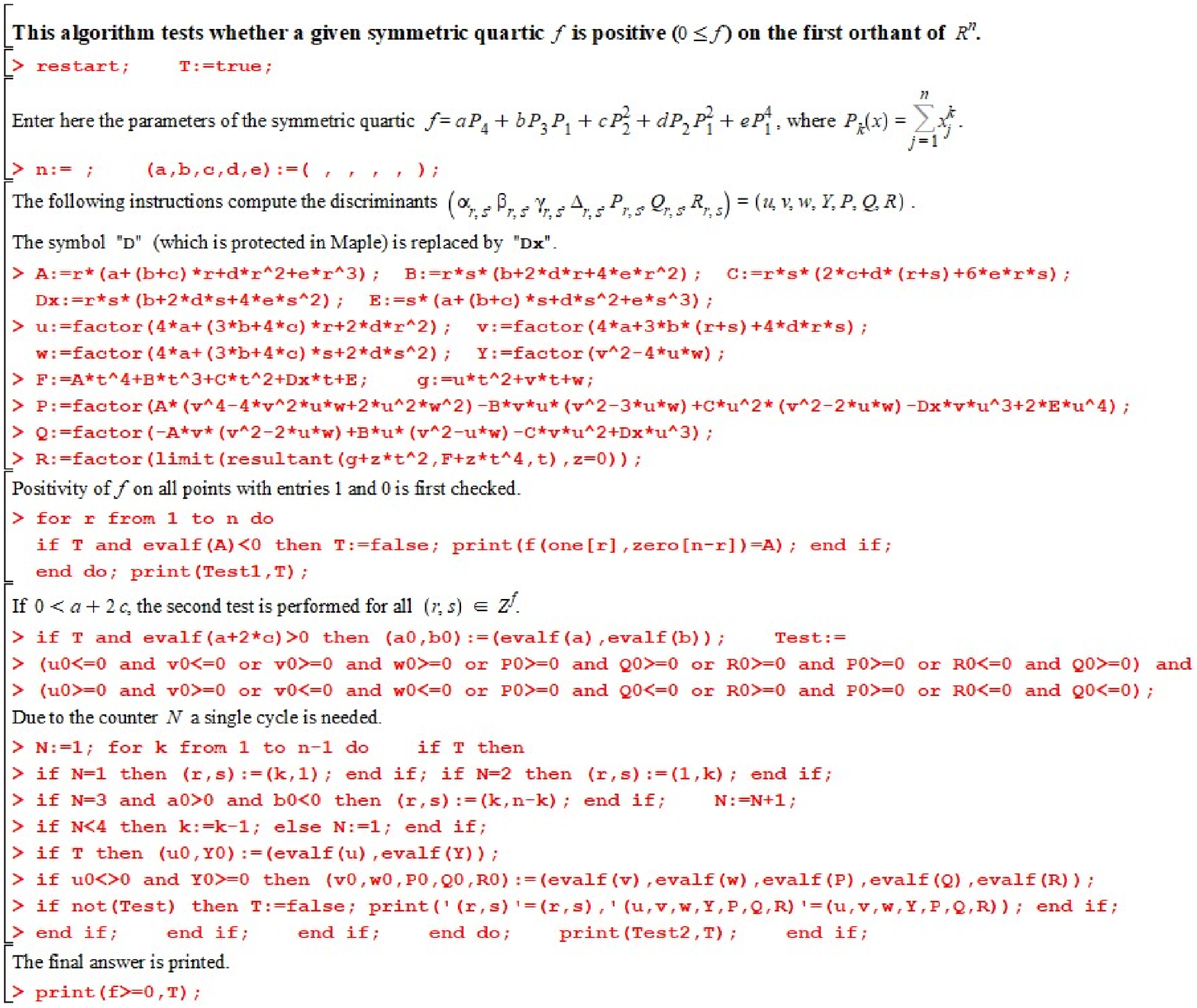} \caption{\footnotesize Maple\,15 implementation of $\mathrm{QE}4_+$.} \label{f.qe4+} \end{center}
\end{figure}

\begin{example}\label{ex.1}
If $f=24P_4-18P_3P_1-8P_2^2+9P_2P_1^2-P_1^4\in{\mathcal H}_4^{[4]}$, then  $f\ge0$ on $\mathbb{R}_+^4$.
\end{example}

\begin{proof}
The algorithm $\mathrm{QE}4_+$ prints the output ``$0\le f,\,\mbox{\it true}$''.
\end{proof}

\begin{example}\label{ex.2}
Let $f=24P_4-19P_3P_1-7P_2^2+9P_2P_1^2-P_1^4\in{\mathcal H}_4^{[n]}$. Then  $f\ge0$ on $\mathbb{R}_+^n$, if and only if $n\in\{2,3\}$.
\end{example}

\begin{proof}
For $n=4$, the algorithm $\mathrm{QE}4_+$ lists the values
\[(r,s) = (1,3),\quad(u,v,w,Y,P,Q,R)=(29,-24,3,228,-108828,14214,-12096),\]
and prints the output ``$0\le f,\,\mbox{\it false}$''. For $n=3$ the output is ``$0\le f,\,\mbox{\it true}$''. Since for integers $1\le p<q$ the first orthant $\mathbb{R}_+^p$ can be identified with $\mathbb{R}_+^p\times\{0_{q-p}\}\subset\mathbb{R}_+^q$, the conclusion follows.
\end{proof}

For symbolic/literal polynomials it may be difficult to decide on the sign of the discriminants, which also are literal expressions. In this case the following suggestions may help (these also apply to the examples from Section~\ref{ss.algorithm2}). Let $D_{r,s}$ denote any of the needed discriminants. Then:
\begin{itemize}
\item If $(r,s)=(k,1)$ or $(r,s)=(1,k)$ for some $k\in\{1,2,\dots,n-1\}$, then in $D_{r,s}$ make the substitutions $k:=u+1$ and $n:=u+p+2$, with $u,p\in\mathbb{N}$.
\item If $(r,s)=(k,n-k)$ for some $k\in\{2,\dots,n-2\}$, then in $D_{r,s}$ make the substitutions $k:=u+2$ and $n:=u+v+4$, with $u,v\in\mathbb{N}$ (see Example~\ref{ex.3}).
\end{itemize}

\begin{example}\label{ex.3}
Let $g=2nP_4-2(n+1)P_3P_1-nP_2^2+(n+3)P_2P_1^2-P_1^4\in{\mathcal H}_4^{[n]}$. Then $g\ge0$ on $\mathbb{R}_+^n$.
\end{example}

\begin{proof}
$g$ satisfies (\ref{e.k}), since\footnote{All needed computations were performed by running the first part of $\mathrm{QE}4_+$.} $g(1_k,0_{n-k})=k(k-1)(k-2)(n-k)\ge0$ for every $k\in\{1,2,\dots,n\}$. Running the first part of $\mathrm{QE}4_+$ gives
\[\left\{\begin{array}{l}P_{k,1}=16kn(n-3)^3(k-1)^5(k-2)[(k-3)n+k+1]\ge0=R_{k,1},\\ \alpha_{1,k}=0,\end{array}\right.\]
and so $(\ref{e.discriminants1})_{k,1}$ and $(\ref{e.discriminants1})_{1,k}$ either hold or need not be satisfied, by Remark~\ref{r.useful}.
If $n\le3$, the conclusion follows at once by Theorem~\ref{t.discriminants1}. Now assume that $n\ge4$, and fix $k\in\{2,\dots,n-2\}$. We have $(k,n-k)=(u+2,v+2)$ for some $u,v\in\mathbb{N}$, with $n=u+v+4$. Running the first part of $\mathrm{QE}4_+$ with these substitutions gives\\[2mm]
$P_{k,n-k}=16(u+2)(u+1)(v+2)(v+1+u)^2(4352+15872v+17408u+25504u^3+28856u^2+60496vu+94648vu^2+12912u^4+22136v^2+80272v^2u+4856v^4+78798vu^3+117100v^2u^2+50798v^3u+37789u^4v+89185u^3v^2+68231v^3u^2+15863v^4u+38625v^2u^4+45801v^3u^3+19405v^4u^2+14744v^3+776v^53964u^5+946u^6+227u^7+40u^8+2vu^9+3u^9+10v^5u^5+2v^6u^4+10835vu^5+2125u^6v+9835v^2u^5+343vu^7+16540v^3u^4+1595v^2u^6+40vu^8+11079v^4u^3+3260v^3u^5+2388v^5u+175v^2u^7+2615v^5u^2+3040v^4u^4+48v^6+360v^3u^6+139v^6u+1205v^5u^3+385v^4u^5+10v^2u^8+20v^3u^7+133v^6u^2+208v^5u^4+20v^4u^6+45v^6u^3)\ge0$,\\[1mm]
$R_{k,n-k}=16(u+2)^2(v+2)^2(v+1+u)^4(u+1)(v+1)(u+v+4)^2(v+u+2)^3(8+8v+8u+10vu+v^2u+vu^2)\ge0$.\\
Hence $(\ref{e.discriminants1})_{k,n-k}$ either holds or needs not be satisfied. By Theorem~\ref{t.discriminants1} we conclude that $\mathrm{QE}_+(g)$ holds.
\end{proof}

If $n\ge4$, then $g$ from Example~\ref{ex.3} also has the following property:
\[\varphi\in{\mathcal H}_4^{[n]}, \ 0\le\varphi\le g\mbox{ on }\mathbb{R}_+^n\Longrightarrow\varphi\in[0,1]\!\cdot\!g,\]
that is, $g$ is extremal (see~\cite[Th.\,16]{timofte3}). This means that the inequality $\mathrm{QE}_+(g)$ holds, but cannot be strengthened in ${\mathcal H}_4^{[n]}$. In the same situation are the polynomials $f$ from the above Example~\ref{ex.1} (see~\cite[Th.\,16]{timofte3}) and $h$ from Example~\ref{ex.4} below (see~\cite[Prop.\,17]{timofte3}). In other words, these are ``good quality'' examples.

\begin{example}\label{ex.4}
Let $h=-n(n-1)P_4+4(n-1)P_3P_1+(n^2-3n+3)P_2^2-2nP_2P_1^2+P_1^4\in{\mathcal H}_4^{[n]}$. Then $h\ge0$ on $\mathbb{R}_+^n$.
\end{example}

\begin{proof}
$h$ satisfies (\ref{e.k}), since $h(1_k,0_{n-k})=k(k-1)(n-k)(n-k-1)\ge0$ for every $k\in\{1,2,\dots,n\}$. By $a=-n(n-1)\le0$ (or by $b=4(n-1)\ge0$), it follows that $Z^h=\{(k,1)|\,1\le k\le n-1\}\cup\{(1,k)|\,1\le k\le n-1\}$. Running the first part of $\mathrm{QE}4_+$ gives
\[\left\{\begin{array}{l}P_{k,1}=256k(k-1)(n-1)(n-3)^3(n-k-1)^5(2nk-n-3k)\ge0=R_{k,1},\\ \alpha_{1,k}=0.\end{array}\right.\]
By Theorem~\ref{t.discriminants1} we conclude that $\mathrm{QE}_+(h)$ holds.
\end{proof}

\section{The problem $\mathrm{QE}(f)$ in ${\mathcal H}_4^{[n]}$}\label{s.QE2}

\subsection{Discriminants for $\mathrm{QE}(f)$}\label{ss.discriminants2}

According to (\ref{e.QE}), for every symmetric quartic $f\in{\mathcal H}_4^{[n]}$ the problem $\mathrm{QE}(f)$ reduces to the quantifier elimination problems
\[f(u\!\cdot\!1_r,v\!\cdot\!1_s)\ge0\quad\forall u,v\in\mathbb{R},\]
considered for all $r,s\in\mathbb{N}^*$, with $r+s=n$. Since $f$ is symmetric, we may restrict the algorithm $\tilde{\mathcal A}$ described in the proof of Theorem~\ref{t.efficient} to run only for $s\ge r\ge1$. For the restricted algorithm ${\mathcal Q}{\mathcal E}4$ obtained in this way the running time is reduced to the half, while the complexity is the same as that of $\tilde{\mathcal A}$. The number of operations performed by ${\mathcal Q}{\mathcal E}4$ on $f$ is
\[N({\mathcal Q}{\mathcal E}4,f)=N(4,2)\left\lfloor\frac n2\right\rfloor.\]
Hence ${\mathcal Q}{\mathcal E}4$ solves $\mathrm{QE}(f)$ in $\mathrm{lin}(n)$ time. As quartics are homogeneous, we see that
\begin{equation}\label{e.equivalence3}\mathrm{QE}(f)\mbox{ holds}\iff\mathrm{QE}(f_{k,n-k})\mbox{ holds for every\, }k\in\{1,\dots,n-1\},\end{equation}
where $f_{k,n-k}(t)=f(t\!\cdot\!1_k,1_{n-k})$ for every $t\in\mathbb{R}$. Therefore, for arbitrary real symbolic polynomial
\begin{equation}\label{e.model2}F(t)=At^4+Bt^3+Ct^2+Dt+E,\,\mbox{ with }\,A>0,\end{equation}
we need to compute explicit discriminants for the univariate problem $\mathrm{QE}(F)$, or for its equivalent form
\begin{equation}\label{e.equivalent}F(t)\ge0\,\mbox{ for every real root $t$ of the derivative }F'.\end{equation}
Solving (\ref{e.equivalent}) depends on the nature of the roots $z_1,z_2,z_3\in\mathbb{C}$ of the polynomial $F'\in\mathbb{R}[t]\subset\mathbb{C}[t]$ (and hence on its discriminant), as well as on the expressions
\begin{equation}\label{e.GHK}\left\{\begin{array}{l}G:=256A^3[F(z_1)+F(z_2)+F(z_3)],\\ H:=128A^3[F(z_1)F(z_2)+F(z_1)F(z_3)+F(z_2)F(z_3)],\\ K:=256A^3F(z_1)F(z_2)F(z_3).\end{array}\right.\end{equation}
Up to a strictly positive factor, the discriminant of $F'$ is
\begin{equation}\label{e.Delta}\Delta:=-108A^2D^2+4AC(27BD-8C^2)-9B^2(3BD-C^2).\end{equation}
In order to compute $G,H,K$, let us observe that for every $y\in\mathbb{R}$, the resultant of $F_y:=F+y$ and $F_y'=F'$ is
\[R(F_y,F')=(4A)^4\prod_{j=1}^3F_y(z_j)=256A^4y^3+A(Gy^2+2Hy+K).\]
Identifying here the coefficients of the polynomials in $y$ leads to
\begin{eqnarray}
\label{e.G}&&G=768A^3E-64A^2(3BD+2C^2)+144AB^2C-27B^4,\\
\label{e.H}&&\!\!\!\begin{array}{ll}H=&\!\!\!384A^3E^2-8A^2(24BDE+16C^2E-9CD^2)\\ &+A(144B^2CE-3B^2D^2-40BC^2D+8C^4)\\ &-B^2(27B^2E-9BCD+2C^3),\end{array}\\
\label{e.resultant2}&&K=\frac{R(F,F')}{A}=\left|\begin{array}{ccccccc}1\,&B&C&D&E&0&0\\ 0\,&A&B&C&D&E&0\\ 0\,&0&A&B&C&D&E\\ 4\,&3B&2C&D&0&0&0\\ 0\,&4A&3B&2C&D&0&0\\ 0\,&0&4A&3B&2C&D&0\\ 0\,&0&0&4A&3B&2C&D\end{array}\right|.
\end{eqnarray}

\begin{notation}\label{n.QE2}
\begin{description}
\item[\rm\bf(i)] The real polynomials $\Delta,G,H,K$ defined by the right-hand members of $(\ref{e.Delta})$--$(\ref{e.resultant2})$ may be considered even if $A=0$.
\item[\rm\bf(ii)] For $F=f_{r,n-r}$, we write the above expressions as $\Delta_r,G_r,H_r,K_r$. For these, the condition $(\ref{e.discriminants2})$ below will be referred to as $(\ref{e.discriminants2})_r$.
\end{description}
\end{notation}

Our next lemma may be viewed as a special case of Descartes' rule of signs.

\begin{lemma}\label{l.Descartes}
Let $u\in\mathbb{R}^m$. Then
\[u\in\mathbb{R}_+^m\iff e_k(u)\ge0\mbox{\, for every\, }k\in\{1,2,\dots,m\},\]
where $e_1, e_2, \dots, e_m$ denote the elementary symmetric functions in $m$ variables.
\end{lemma}

\begin{proof}
We only need to prove ``$\Leftarrow$''. If $t:=\min_{1\le j\le m}u_j<0$, then the hypothesis yields $0<(-t)^m+\sum_{k=1}^me_k(u)(-t)^{m-k}=\prod_{k=1}^m(u_k-t)=0$, which is absurd. Hence $u\in\mathbb{R}_+^m$.
\end{proof}

The following needed result is a version of~\cite[Th.\,1]{p.r}, however, it is much easier to give a direct proof than to derive it from the cited result.

\begin{theorem}\label{t.tool}
Let a polynomial $F$ as in $(\ref{e.model2})$. Then $\mathrm{QE}(F)$ is equivalent to
\begin{equation}\label{e.discriminants2}K\ge0>\Delta\mbox{ \bf or }G,H,K\ge0.\end{equation}
\end{theorem}

\begin{proof}
As $A\ne0$, we have $\deg(F)=4$ and $\deg(F')=3$. Let $z_1\in\mathbb{R}$ and $z_2,z_3\in\mathbb{C}$ denote the roots of $F'$.\\
``$\Rightarrow$''. If $\Delta<0$, then $z_3=\bar z_2\in\mathbb{C}\setminus\mathbb{R}$, and so $F(z_1)\ge0,F(z_3)=\overline{F(z_2)}$. We thus get $K=256A^3F(z_1)|F(z_2)|^2\ge0>\Delta$. If $\Delta\ge0$, then $z_1,z_2,z_3\in\mathbb{R}$, and so $F(z_1),F(z_2),F(z_3)\ge0$. Since $A>0$, this leads by (\ref{e.GHK}) to $G,H,K\ge0$.\\
``$\Leftarrow$''. Suppose (\ref{e.equivalent}) false. There is no loss of generality in assuming that $F(z_1)<0$. As $K\ge0$, we have $F(z_2)F(z_3)\le0$. We next analyze two cases.\\
\emph{Case 1.} If $\Delta<0$, then $F(z_3)=\overline{F(z_2)}$ leads to $|F(z_2)|^2=F(z_2)F(z_3)\le0$, which yields $F(z_2)=F(z_3)=0$. Hence both $z_2,z_3\in\mathbb{C}\setminus\mathbb{R}$ are repeating roots of $F$, and so $F$ has no real roots. We thus get $F>0$ on $\mathbb{R}$, a contradiction.\\
\emph{Case 2.} If $\Delta\ge0$, we must have $G,H,K\ge0$. Since $(F(z_1),F(z_2),F(z_3))\in\mathbb{R}^3$, by Lemma~\ref{l.Descartes} it follows that $F(z_1)\ge0$, a contradiction.\\
We conclude that $F$ satisfies (\ref{e.equivalent}), that is, $\mathrm{QE}(F)$ holds.
\end{proof}

\begin{theorem}[$\mathrm{QE}4$ algorithm]\label{t.discriminants2}
Let $f\in{\mathcal H}_4^{[n]}$. Then $\mathrm{QE}(f)$ is equivalent to
\[\left\{\begin{array}{l}(\ref{e.k})\mbox{ holds for }f,\\ (\ref{e.discriminants2})_r\mbox{ holds for every }r\in\{1,2,\dots,n-1\}.\end{array}\right.\]
\end{theorem}

\begin{proof}
To simplify notation, we will write the polynomial $f_{r,n-r}$ and its coefficients $A_{r,n-r},\dots,E_{r,n-r}$ as $f_r$ and $A_r,\dots,E_r$.\\
``$\Rightarrow$''. Clearly, $\mathrm{QE}(f)$ yields (\ref{e.k}). Let us fix $r\in\{1,2,\dots,n-1\}$. As $\mathrm{QE}(f_r)$ holds, we have $A_r\ge0$ and $f_r$ has even degree (or $f_r\equiv0$). We next analyze two cases.\\
\emph{Case 1.} If $\deg(f_r)\le2$, then $A_r=B_r=0$ leads by (\ref{e.G})--(\ref{e.resultant2}) to $G_r=H_r=K_r=0$. Hence $(\ref{e.discriminants2})_r$ holds.\\
\emph{Case 2.} If $\deg(f_r)=4$, then $A_r>0$. By Theorem~\ref{t.tool}, $\mathrm{QE}(f_r)$ yields $(\ref{e.discriminants2})_r$.\\
From the above cases we conclude that $(\ref{e.discriminants2})_r$ holds for every $r\in\{1,2,\dots,n-1\}$.\\
``$\Leftarrow$''. Fix $r\in\{1,2,\dots,n-1\}$. Let us observe that
\begin{equation}\label{e.switch2}\left\{\begin{array}{l}(A_{n-r},B_{n-r},C_{n-r},D_{n-r},E_{n-r})=(E_r,D_r,C_r,B_r,A_r),\\ f_s(t)=t^4f_r\left(\frac1t\right)\mbox{\, for every\, }t\in\mathbb{R}\setminus\{0\}.\end{array}\right.\end{equation}
By (\ref{e.k}), we get $A_r=f(1_r,0_{n-r})\ge0$ and, similarly, $A_s\ge0$. There are three cases.\\
\emph{Case 1.} If $A_r>0$, then $\mathrm{QE}(f_r)$ holds, by $(\ref{e.discriminants2})_r$ and Theorem~\ref{t.tool}.\\
\emph{Case 2.} If $A_s>0$, as in the previous case it follows that $\mathrm{QE}(f_s)$ holds, and consequently so does $\mathrm{QE}(f_r)$, by (\ref{e.switch2}).\\
\emph{Case 3.} If $A_r=E_r=A_s=0$, some easy computations lead by (\ref{e.Delta}),\,(\ref{e.G}),\,(\ref{e.resultant2}), to
\begin{equation}\label{e.absurd}\left\{\begin{array}{l}\Delta_r=9B_r^2(C_r^2-3B_rD_r),\\ G_r=-27B_r^4,\quad K_r=B_r^2D_r^2(C_r^2-4B_rD_r).\end{array}\right.\end{equation}
As $(\ref{e.discriminants2})_r$ holds, we have $K_r\ge0$, and so $C_r^2\ge4B_rD_r$, which leads by (\ref{e.absurd}) to $\Delta_r\ge\frac{9}{4}B_r^2C_r^2\ge0$. Since $(\ref{e.discriminants2})_r$ holds, we must have $G_r\ge0$, which forces $B_r=0$, by (\ref{e.absurd}). A similar argument shows that $D_r=B_{n-r}=0$. By (\ref{e.k}) it follows that $C_r=f_r(1)=f(1_{r+s},0_{n-r-s})\ge0$, hence that
$f_r(t)=C_rt^2\ge0$ for every $t\in\mathbb{R}$.\\
From the above three cases, $\mathrm{QE}(f_r)$ holds for every $r\in\{1,2,\dots,n-1\}$. By Theorem~\ref{t.quartics} we conclude that $\mathrm{QE}(f)$ holds.
\end{proof}

\subsection{The algorithm $\mathrm{QE}4$ and examples}\label{ss.algorithm2}

Theorem~\ref{t.discriminants2} leads to the algorithm $\mathrm{QE}4$ for solving $\mathrm{QE}(f)$ (see Figure~\ref{f.qe4}). The algorithm performs $n$ tests for (\ref{e.k}) and $n-1$ sets of computations/tests (of the \emph{same complexity} for all $r\in\{1,\dots,n-1\}$) for the conditions $(\ref{e.discriminants2})_r$.

\begin{figure}\begin{center}
\includegraphics[width=12cm]{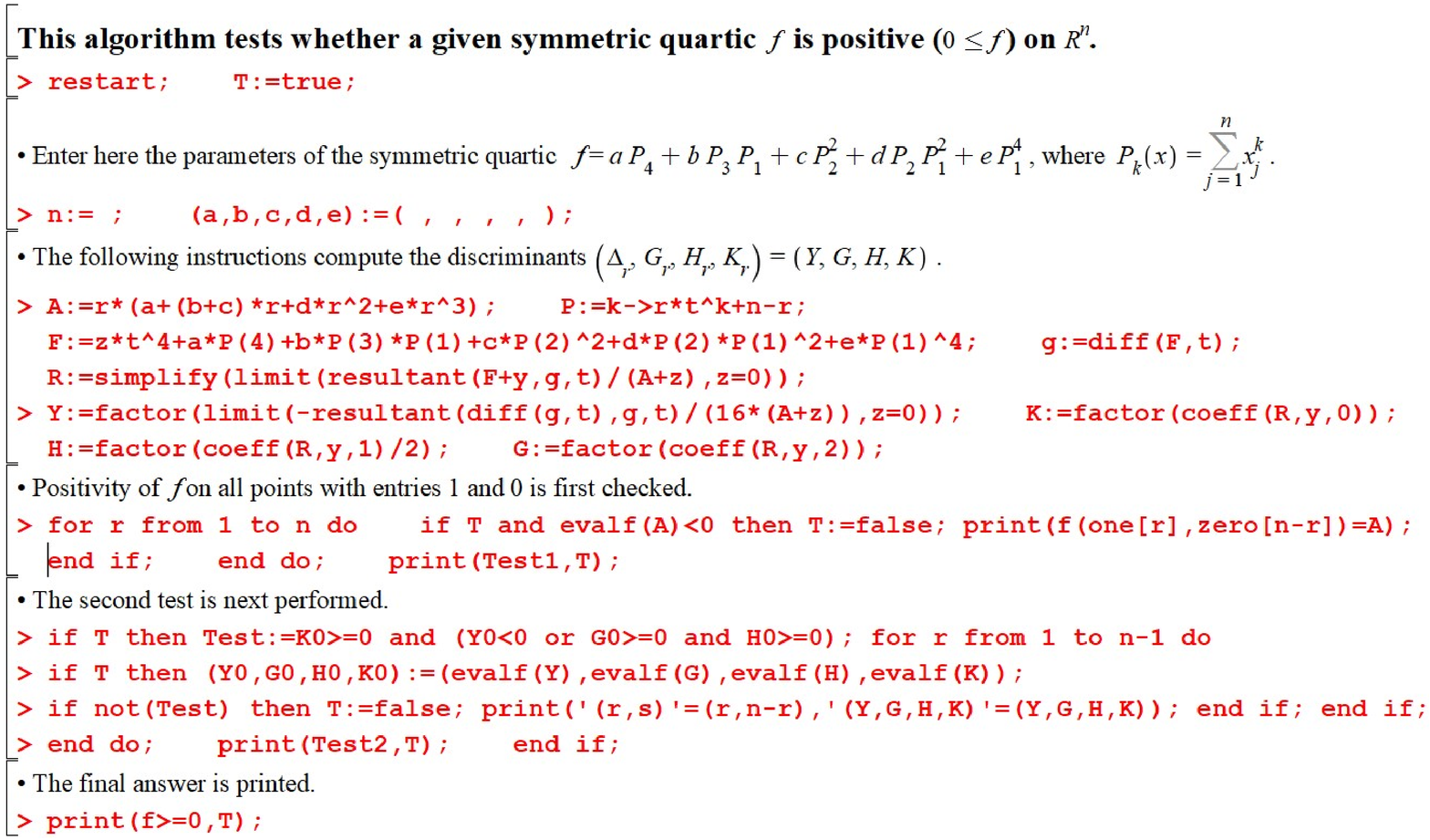}
\caption{\footnotesize Maple\,15 implementation of $\mathrm{QE}4$.}
\label{f.qe4}
\end{center}\end{figure}

\begin{example}\label{ex.5}
Let $f=-2(n-1)P_3P_1+(n-2)P_2^2+(n+1)P_2P_1^2-P_1^4\in{\mathcal H}_4^{[n]}$. Then\footnote{This is one of Newton's inequalities.} $f\ge0$ on $\mathbb{R}^n$.
\end{example}

\begin{proof}
$f$ satisfies (\ref{e.k}), since $f(1_k,0_{n-k})=k^2(k-1)(n-k)\ge0$ for every $k\in\{1,2,\dots,n\}$. For $r\in\{1,2,\dots,n-1\}$, we have $(r,n-r)=(u+1,v+1)$ for some $u,v\in\mathbb{N}$, with $n=u+v+2$. Running the first part of $\mathrm{QE}4$ with these substitutions gives
\begin{eqnarray*}
G_r\!\!\!&=&\!\!\!16u^2(u+1)^4(v+1)^4(u+v+1)^2(u^2+2u^3v+u^2v^2+22u^2v\\ &&+20uv+2u^3+u^4+20uv^2-8v^2)\ge0,\\
H_r\!\!\!&=&\!\!\!8uv(u+1)^5(v+1)^5(u+v)^3(u+v+1)^4\ge0,\\
K_r\!\!\!&=&\!\!\!0.
\end{eqnarray*}
By Theorem~\ref{t.discriminants2}, we conclude that $\mathrm{QE}(f)$ holds.
\end{proof}

\begin{example}\label{ex.6}
Let $h\in{\mathcal H}_4^{[n]}$ as in Example~\ref{ex.4}. Then $h\ge0$ on $\mathbb{R}^n$.
\end{example}

\begin{proof}
That $h$ satisfies (\ref{e.k}) was proved in Example~\ref{ex.4}. Running the first part of $\mathrm{QE}4$ gives $G_r=H_r=K_r=0$. By Theorem~\ref{t.discriminants2}, we conclude that $\mathrm{QE}(h)$ holds.
\end{proof}

\section{Numerical tests and conclusions}\label{s.tests}

In this section we describe some numerical tests performed\footnote{With Maple\,15 on a Dell Inspiron 5570 with 16GB RAM and Windows 10 OS.} in order to evaluate the efficiency of our algorithms.

Let us note that both $\mathrm{QE}4_+$ and $\mathrm{QE}4$ print the final output ``$0\le f,\,\mbox{\it false}$'', as soon as the variable $T$ changes from ``{\it true}'' to ``{\it false}''. Therefore, testing valid inequalities takes longer. In order to avoid quick negative outputs, the numerical coefficients $a,\dots,e$ as in (\ref{e.rep}) are chosen such that $\mathrm{QE}_+(f)$ (or $\mathrm{QE}(f)$) holds for every $n\ge2$. We thus compare in similar conditions the running time (seconds) and the amount of memory (MB) required for the execution of the algorithms, for several values of $n$. The results of our tests (all with final output ``$0\le f,\,\mbox{\it true}$'') are listed below.\vspace{5mm}

Algorithm $\mathrm{QE}4_+$: test for $(a,b,c,d,e)=(-6,8,3,-6,1)$.\\ \emph{Remark: only the first test is performed, since $a+2c\le0$.}\\
\begin{tabular}{c|c|c|c|c|c|c|}
$\frac{n}{1000000}=$ & $10$ & $20$ & $30$ & $40$ & $50$ & $100$\\ \hline
running time (sec.) & $18.82$ & $37.39$ & $57.84$ & $77.03$ & $95.32$ & $196.68$\\ \hline
memory (MB) & $49.42$ & $49.42$ & $49.42$ & $49.42$ & $49.42$ & $49.42$\\ \hline
\end{tabular}\vspace{5mm}

Algorithm $\mathrm{QE}4_+$: test for $(a,b,c,d,e)=(6,-4,-1,1,0)$.\\ \emph{Remark: both test are performed, since $a+2c>0$.}\\
\begin{tabular}{c|c|c|c|c|c|c|}
$\frac{n}{1000000}=$ & $0.5$ & $0.6$ & $0.7$ & $0.8$ & $0.9$ & $5$\\ \hline
running time (sec.) & $27.89$ & $32.93$ & $39.07$ & $44.57$ & $50.28$ & $283.14$\\ \hline
memory (MB) & $68.05$ & $68.05$ & $68.05$ & $68.05$ & $68.05$ & $68.05$\\ \hline
\end{tabular}\vspace{5mm}

Algorithm $\mathrm{QE}4_+$: test for $(a,b,c,d,e)=(-2,\frac{52+20\sqrt{10}}{27},\frac{83-20\sqrt{10}}{27},-4,1)$.\\ \emph{Remark: $a+2c>0$ and the coefficients are not all rational.}\\
\begin{tabular}{c|c|c|c|c|c|c|}
$\frac{n}{1000000}=$ & $10$ & $20$ & $30$ & $40$ & $50$ & $100$\\ \hline
running time (sec.) & $19.39$ & $38.96$ & $58.59$ & $80.81$ & $96.89$ & $194.81$\\ \hline
memory (MB) & $47.99$ & $47.99$ & $47.99$ & $47.99$ & $47.99$ & $47.99$\\ \hline
\end{tabular}\vspace{5mm}

Algorithm $\mathrm{QE}4$: test for $(a,b,c,d,e)=(0,-2,1,1,0)$.\\ \emph{Remark: none.}\\
\begin{tabular}{c|c|c|c|c|c|c|}
$\frac{n}{1000000}=$ & $1$ & $2$ & $3$ & $4$ & $5$ & $10$\\ \hline
running time (sec.) & $13.96$ & $27.84$ & $42.35$ & $56.12$ & $70.68$ & $141.46$\\ \hline
memory (MB) & $68.17$ & $68.17$ & $68.17$ & $68.17$ & $68.17$ & $68.23$\\ \hline
\end{tabular}\vspace{5mm}

For fixed numerical coefficients, the running time is almost linear in $n$, just as expected. For $f\in{\mathcal H}_4^{[n]}$, both $\mathrm{QE}_+(f)$ and $\mathrm{QE}(f)$ are efficiently solved by our algorithms. Even though both $\mathrm{QE}4_+$ and $\mathrm{QE}4$ are designed for numerical examples, they may assist in proving symbolic inequalities, as shown in Sections~\ref{ss.algorithm1} and~\ref{ss.algorithm2}.

We next indicate a method for finding the coefficients from (\ref{e.rep}). Since a given $f\in{\mathcal H}_4^{[n]}$ is usually expressed as a linear combination of monomial symmetric polynomials, assume that
\begin{equation}\label{e.msp}f=\alpha M_4+\beta M_{3,1}+\gamma M_{2,2}+\delta M_{2,1,1}+\varepsilon M_{1,1,1,1}\quad(\alpha,\beta,\gamma,\delta,\varepsilon\in\mathbb{R}).\end{equation}
Here, $M_{\lambda_1,\dots,\lambda_k}\in\mathbb{R}[x_1,\dots,x_n]$ is the sum of all distinct monomials $x_{i_1}^{\lambda_1}\cdots x_{i_k}^{\lambda_k}$, with distinct $i_1,\dots,i_k\in\{1,\dots,n\}$ (by convention, the sum vanishes if $k>n$). An easy computation shows that switching from (\ref{e.msp}) to the representation (\ref{e.rep}) may be done by using the equality
\[\left(\begin{array}{c}a\\b\\c\\d\\e\end{array}\right)=\left(\begin{array}{rrrrr} 1&-1&-\frac12&1&-\frac14\\ 0&1&0&-1&\frac13\\ 0&0&\frac12&-\frac12&\frac18\\ 0&0&0&\frac12&-\frac14\\ 0&0&0&0&\frac1{24} \end{array}\right)\cdot\left(\begin{array}{c}\alpha\\ \beta\\ \gamma\\ \delta\\ \varepsilon\end{array}\right).\]
Thus both algorithms may be modified in order to accept as input the coefficients $\alpha,\dots,\varepsilon$ from (\ref{e.msp}).

\end{document}